\documentclass[11pt]{amsart}
\usepackage{amsmath,amssymb,yhmath,url,color,mathtools,braket,mathrsfs,enumerate,tikz,lscape}
\usepackage[text={144mm,220mm},centering]{geometry}

\makeatletter
\def\subsection{\@startsection{subsection}{2}%
  \z@{.5\linespacing\@plus.7\linespacing}{.3\linespacing}%
  {\normalfont\bfseries}}
\makeatother

\newtheorem{lemma}{Lemma}[section]
\newtheorem{prop}[lemma]{Proposition}
\newtheorem{cor}[lemma]{Corollary}
\newtheorem{thm}[lemma]{Theorem}

\newtheorem{thm?}[lemma]{Theorem?}

\newtheorem{conj}[lemma]{Conjecture}

\newtheorem*{unthm}{Theorem}
\newtheorem{fact}{Fact}

\theoremstyle{definition}
\newtheorem{example}[lemma]{Example}

\theoremstyle{remark}
\newtheorem{remark}[lemma]{Remark}

\begin{document}
\title{Typically Bounding Torsion}
\author{Pete L. Clark}
\author{Marko Milosevic}
\author{Paul Pollack}

\date{\today}


%

\newcommand{\etalchar}[1]{$^{#1}$}
\newcommand{\F}{\mathbb{F}}
\newcommand{\et}{\textrm{\'et}}
\newcommand{\ra}{\ensuremath{\rightarrow}}
\newcommand{\FF}{\F}
\newcommand{\Z}{\mathbb{Z}}
\newcommand{\N}{\mathcal{N}}
\newcommand{\ch}{}
\newcommand{\R}{\mathbb{R}}
\newcommand{\PP}{\mathbb{P}}
\newcommand{\pp}{\mathfrak{p}}
\newcommand{\C}{\mathbb{C}}
\newcommand{\Q}{\mathbb{Q}}
\newcommand{\tpqr}{\widetilde{\triangle(p,q,r)}}
\newcommand{\ab}{\operatorname{ab}}
\newcommand{\Aut}{\operatorname{Aut}}
\newcommand{\gk}{\mathfrak{g}_K}
\newcommand{\gq}{\mathfrak{g}_{\Q}}
\newcommand{\OQ}{\overline{\Q}}
\newcommand{\Out}{\operatorname{Out}}
\newcommand{\End}{\operatorname{End}}
\newcommand{\Gon}{\operatorname{Gon}}
\newcommand{\Gal}{\operatorname{Gal}}
\newcommand{\CT}{(\mathcal{C},\mathcal{T})}
\newcommand{\ttop}{\operatorname{top}}
\newcommand{\lcm}{\operatorname{lcm}}
\newcommand{\Div}{\operatorname{Div}}
\newcommand{\OO}{\mathcal{O}}
\newcommand{\rank}{\operatorname{rank}}
\newcommand{\tors}{\operatorname{tors}}
\newcommand{\IM}{\operatorname{IM}}
\newcommand{\CM}{\operatorname{CM}}
\newcommand{\Frac}{\operatorname{Frac}}
\newcommand{\Pic}{\operatorname{Pic}}
\newcommand{\coker}{\operatorname{coker}}
\newcommand{\Cl}{\operatorname{Cl}}
\newcommand{\loc}{\operatorname{loc}}
\newcommand{\GL}{\operatorname{GL}}
\newcommand{\PGL}{\operatorname{PGL}}
\newcommand{\SL}{\operatorname{SL}}
\newcommand{\PSL}{\operatorname{PSL}}
\newcommand{\Frob}{\operatorname{Frob}}
\newcommand{\Hom}{\operatorname{Hom}}
\newcommand{\Coker}{\operatorname{\coker}}
\newcommand{\Ker}{\ker}
\renewcommand{\gg}{\mathfrak{g}}
\newcommand{\sep}{\operatorname{sep}}
\newcommand{\new}{\operatorname{new}}
\newcommand{\Ok}{\mathcal{O}_K}
\newcommand{\ord}{\operatorname{ord}}
\newcommand{\mm}{\mathfrak{m}}
\newcommand{\Ohell}{\OO_{p^{\infty}}}
\newcommand{\ff}{\mathfrak{f}}
\renewcommand{\N}{\mathbb{N}}
\newcommand{\Gm}{\mathbb{G}_m}
\renewcommand{\P}{\mathbb{P}}
\newcommand{\ffield}{\ensuremath{\mathbb{F}_p}}
\newcommand{\galoisrep}{\ensuremath{\rho_{E,p}}}
\newcommand{\GLp}{\ensuremath{\text{GL}_2(\ffield)}}
\newcommand{\torsion}{\ensuremath{E(F)[\text{tors}]}}
\newcommand{\Rdegree}{\ensuremath{\displaystyle [F(R):F]}}
\newcommand{\splitcartan}{\ensuremath{\mathcal{C}_{\text{sp}}}}
\newcommand{\nonsplitcartan}{\ensuremath{\mathcal{C}_{\text{ns}}}}
\newcommand{\Nsplitcartan}{\ensuremath{\mathcal{C}_{\text{sp}}^+}}
\newcommand{\Nnonsplitcartan}{\ensuremath{\mathcal{C}_{\text{ns}}^+}}
\renewcommand{\thefootnote}{\arabic{footnote}}
\newcommand{\divides}{\bigm|}
\begin{abstract}
We formulate the notion of \emph{typical boundedness} of torsion on a family of abelian varieties defined over number fields.  This means that the torsion subgroups of
elements in the family can be made uniformly bounded by removing from the family all abelian varieties defined over number fields
of degree lying in a set of arbitrarily small density.  We show that for each fixed $g$, torsion is typically bounded on the family of all $g$-dimensional CM abelian varieties.  We show that torsion is \emph{not} typically bounded on the family of all elliptic curves, and we establish results -- some unconditional and some conditional -- on typical boundedness of torsion of elliptic curves for which the degree of the $j$-invariant is fixed.
\end{abstract}

\maketitle


\section{Introduction}

\subsection{Notation and terminology}
\noindent For a subset $\mathcal{S} \subset \Z^+$,
we define the \emph{upper density}
\[ \overline{\delta}(\mathcal{S}) \coloneqq \limsup_{x \ra \infty} \frac{\# (\mathcal{S} \cap [1,x])}{x} \]
and the \emph{lower density}
\[ \underline{\delta}(\mathcal{S}) \coloneqq \liminf_{x \ra \infty} \frac{\# (\mathcal{S} \cap [1,x])}{x}. \]
When $\overline{\delta}(\mathcal{S}) = \underline{\delta}(\mathcal{S})$, we denote it by $\delta(\mathcal{S})$
and call it the \emph{asymptotic density} of $\mathcal{S}$.
\\ \\
For a field $F$, let $F^{\sep}$ be a separable algebraic closure of $F$ and let $\gg_F \coloneqq \Aut(F^{\sep}/F)$ be the absolute Galois
group of $F$.  For an abelian variety $A_{/F}$, we write $\End A$ for the endomorphism ring of $A_{/F^{\sep}}$ and
$\End^0 A \coloneqq \End A \otimes_{\Z} \Q$, both viewed as $\gg_F$-modules.  An abelian variety $A$ has \textbf{complex multiplication (CM)} if $\End^0 A$ contains an \'etale $\Q$-algebra (a finite product of number fields) of degree $2 \dim A$.

\subsection{Typical boundedness: a motivating result}
\noindent Let $T_{\CM}(g,d)$ denote the supremum of $\# A(F)[\tors]$ as $F$ ranges over degree $d$ number fields
and $A$ ranges over all CM abelian varieties defined over $F$.  Silverberg \cite{Silverberg88, Silverberg92a} showed that
$T_{\CM}(g,d) < \infty$ for all $g$ and $d$ and obtained explicit upper bounds.  Sharper results were attained for CM elliptic curves by  Silverberg and also
recently by several others.   Clark-Pollack showed \cite{TRUTH, TRUTHII} that
\begin{equation}\label{eq:truthIIeq} \limsup_{d \ra \infty} \frac{T_{\CM}(1,d)}{d \log \log d} = \frac{e^{\gamma} \pi}{\sqrt{3}}. \end{equation}
Thus the \emph{upper order} of $T_{\CM}(1,d)$ is now known.  Work of Bourdon, Clark and Pollack \cite{BCP17} viewed $T_{\CM}(d)$ as an ``arithmetic function'' and studied other aspects
of its distribution, in particular giving the following \emph{normal order} result.

\begin{thm}[{\cite[Thm. 1.1a)]{BCP17}}]
\label{MOTIVETHM}
For all $\epsilon > 0$, there is $B_{\epsilon} \in \Z^+$ such that
\[ \overline{\delta}( \{d \in \Z^+ \mid T_{\CM}(1,d) \geq B_{\epsilon}\}) \leq \epsilon. \]
\end{thm}

\subsection{Typical boundedness: definition and first examples}
\noindent
By a \emph{family of abelian varieties}, we mean a class $\mathcal{F}$ each of whose elements is an abelian variety
defined over a number field. (We will only consider classes such that if $(A_1)_{/F}$ lies in $\mathcal{F}$ then any $F$-isomorphic
abelian variety $(A_2)_{/F}$ also lies in $\mathcal{F}$.) Thus for instance we are permitted to take $\mathcal{F}$ to be the family of all abelian varieties over all number
fields.
We say that \emph{torsion is typically bounded on} $\mathcal{F}$ if for all $\epsilon > 0$,
there is $B_{\epsilon} \in \Z^+$ such that the set
\[ \mathcal{S}(\mathcal{F},B_{\epsilon}) \coloneqq \{ d \in \Z^+ \mid \exists \  A_{/F} \in \mathcal{F} \text{ such that }
[F:\Q] = d \text{ and } \# A(F)[\tors] \geq B_{\epsilon} \} \]
has upper density at most $\epsilon$.

\begin{remark}
\label{ABBEYERMARK}
If $\mathcal{F}_1,\ldots,\mathcal{F}_n$ are families of abelian varieties, then torsion is typically bounded on $\mathcal{F}_i$
for all $1 \leq i \leq n$ iff torsion is typically bounded on $\bigcup_{i=1}^n \mathcal{F}_i$.
\end{remark}
\noindent
We can rephrase Theorem \ref{MOTIVETHM} as follows:

\begin{thm}
\label{REMOTIVETHM}
Torsion is typically bounded on the family $\mathcal{A}_{\CM}(1)$ of all CM elliptic curves (over all number fields).
\end{thm}
\noindent
In this paper we study typical boundedness of torsion in several other natural families of abelian varieties.  The following
example shows that certain kinds of families must be avoided for rather trivial reasons.

\begin{example}
\label{EXAMPLE1}
Torsion is \emph{not} typically bounded on the family $\mathcal{A}_{\CM}$ of all CM abelian varieties (over all number fields).  Indeed, we have $T_{\CM}(1,1) = 6$ \cite{Olson74}, so there is a CM elliptic curve $E_{/\Q}$ with $\# E(\Q)[\tors] = 6$.
Then for all $g \in \Z^+$, the abelian variety $E^g_{/\Q}$ has CM and $\# E^g(\Q)[\tors] = 6^g$.  Thus
for every number field $F$ we have $\# E^g_{/F}(F)[\tors] \geq 6^g$, so for all $B \in \Z^+$ we have $\mathcal{S}(\mathcal{A}_{\CM},B) = \Z^+$.
\end{example}
\noindent
Example \ref{EXAMPLE1} is not really about complex multiplication.  Rather it shows: if a
family $\mathcal{F}$ is closed under taking powers of abelian varieties and contains an element with nontrivial
torsion subgroup, then torsion is not typically bounded on $\mathcal{F}$.

\begin{remark}
Consider the family $\mathcal{F}$ of all geometrically simple abelian varieties $A$ (over all number fields).  It ought to be true
that $\sup_{A_{/\Q} \in \mathcal{F}} \# A(\Q)[\tors] = \infty$ -- and thus torsion is not typically bounded on $\mathcal{F}$.  But this seems to be an open problem.
\end{remark}

\subsection{Typical boundedness: main results}
\noindent
Our first result is a direct generalization of Theorem \ref{REMOTIVETHM}.

\begin{thm}
\label{BIGCMTHM}
For all $g \in \Z^+$, torsion is typically bounded on the family $\mathcal{A}_{CM}(g)$ of all $g$-dimensional abelian varieties
(over all number fields).
\end{thm}
\noindent
The next result was communicated to us by Filip Najman.

\begin{thm}
\label{NAJMANTHM}
Let $\mathcal{A}(1)$ be the family of all elliptic curves (over all number fields).  For each $B \in \Z^+$, the set
$\mathcal{S}(\mathcal{A}(1),B)$ contains all but finitely many positive integers.  Thus torsion is \emph{not} typically
bounded on $\mathcal{A}(1)$.
\end{thm}
\noindent
So we must restrict the family $\mathcal{A}(1)$ in some way to get typical boundedness.   Our remaining results address this.

\begin{thm}
\label{FZEROTHM}
Let $F_0$ be a number field that does not contain the Hilbert class field of any imaginary quadratic field.  If $[F_0:\Q] \geq 3$,
we assume the Generalized Riemann Hypothesis (GRH).  Then torsion is typically bounded on the family $\mathcal{E}_{F_0}$
of all elliptic curves $E$ defined over a number field $F \supset F_0$ such that $j(E) \in F_0$.
\end{thm}
\noindent
For $d \in \Z^+$, we introduce a hypothesis $\operatorname{SI}(d_0)$ defined as follows:
\[ \operatorname{SI}(d_0)\!:\quad\text{\begin{minipage}[t]{30em}
There is prime $\ell_0 = \ell_0(d_0)$ such that for all primes $\ell > \ell_0$, the modular curve $X_0(\ell)$ has 
no noncuspidal non-CM points of degree $d_0$. 
\end{minipage}} \]

\begin{remark}
\label{REMARK1.8} When $d_0=1$, Mazur showed that $\operatorname{SI}(d_0)$ holds and that the optimal value of $\ell_0$ is $37$ \cite{Mazur78}.
Whether $\operatorname{SI}(d_0)$ holds for any $d_0 \geq 2$ is not known, though it is a folk conjecture that $\operatorname{SI}(d_0)$
holds for all $d_0 \in \Z^+$.
\end{remark}

\begin{thm}
\label{DZEROTHM}
Let $d_0 \in \Z^+$.  If $\operatorname{SI}(d_0)$ holds, then torsion is typically bounded on the family $\mathcal{E}_{d_0}$
of all elliptic curves $E_{/F}$ defined over a number field $F$ such that $[\Q(j(E)):\Q] = d_0$.
\end{thm}

\begin{remark}
This paper is cognate to another work \cite{POLYSKETCH}, written in parallel, giving upper bounds on $\# E(F)[\tors]$ for an elliptic curve $E/F$ that are polynomial in $[F:\Q]$,
under the same hypotheses as Theorems \ref{FZEROTHM} and \ref{DZEROTHM}.
\end{remark}

\subsection{The plan of the paper}
\noindent
In $\S$2 we deduce Theorem \ref{NAJMANTHM} from some general results on degrees of closed
points of varieties over Hilbertian fields.  Our main tool is the Riemann-Roch theorem.
\\ \indent
In $\S$3 we present a formalism for proving typical boundedness of families of abelian varieties, reducing the task
to showing divisibilities on the degrees of the field of definition of a torsion point of order a large power of a fixed prime -- Condition (P1) -- and
of the field of definition of a prime order torsion point -- Condition (P2).  We show a relationship between strong uniform
boundedness results and Condition (P1), reducing the proofs of Theorems \ref{BIGCMTHM}, \ref{FZEROTHM} and
\ref{DZEROTHM} to verifying Condition (P2).
\\ \indent
In $\S$4 we verify Condition (P2) for $\mathcal{A}_{\operatorname{CM}}(g)$.  We use a recent result of Gaudron-R\'emond \cite{GR17} that builds on important work of Silverberg \cite{Silverberg88, Silverberg92a}.
\\ \indent
In $\S$5 we verify Condition (P2) for the families $\mathcal{E}_{F_0}$ and $\mathcal{E}_{d_0}$ under the assumptions in Theorems \ref{FZEROTHM} and \ref{DZEROTHM}, thus proving those results.  Our general strategy is heavily influenced by work of
Lozano-Robledo \cite{LR13} that gives lower bounds on the degree of the field of definition of a point of prime order $p$
for an elliptic curve defined over $\Q$.  Whereas Lozano-Robledo used the complete classification of rational isogenies on
elliptic curves over $\Q$, we instead use finiteness theorems of Mazur, Momose and Larson-Vaintrob.

\subsubsection*{Acknowledgments}
Our proof of Theorem \ref{BIGCMTHM} draws from results in the literature that were brought to our attention by Abbey Bourdon.  As mentioned above, Theorem \ref{NAJMANTHM} was communicated to us by Filip Najman.  This result has become a key part of the narrative
of our paper, and we are very grateful to him.  Theorem \ref{DZEROTHM} is part of the thesis work of the second author under the
direction of the first author.

\section{The degrees of closed points on an algebraic variety}
\noindent
In this section we prove Theorem \ref{NAJMANTHM}.  As observed by Najman, the key observation is that a curve defined over a
number field $k$ that has a $k$-rational point must have infinitely many closed points of degree $d$ for all sufficiently large $d$, applied to the modular curves $X_1(N)_{/\Q}$.  In fact we will prove a generalization of Najman's observation to closed points on varieties over any Hilbertian field, Theorem \ref{MONSTERNAJMANTHM}.

\medskip\noindent
Let $k$ be a field.  A \textbf{variety} $V_{/k}$ is a finite-type $k$-scheme.  By a \textbf{nice variety}
$V_{/k}$ we mean a variety that is smooth, projective and geometrically integral.

\medskip\noindent
The natural context of the results of this section is the class of  \emph{Hilbertian fields} (see e.g. \cite{Fried-Jarden}), namely those for which the conclusion of
Hilbert's irreducibility theorem holds.  For our purposes here it suffices to know that number fields are Hilbertian \cite[Ch. 13]{Fried-Jarden}: Hilbert's theorem.

\subsection{Curves with a $k$-rational point}
\begin{thm}
\label{THM1}
Let $k$ be a Hilbertian field, let $C_{/k}$ be a nice curve of genus $g$, and let $P \in C(k)$.  Then there is $D \in \Z^+$ such that for all $d \geq D$, the curve $C$ has infinitely many closed points of degree $d$.  More precisely:
\begin{itemize}
\item[a)] If $g = 0$, we may take $D = 1$.
\item[b)] If $g \geq 1$, we may take $D = 2g$.  If $P$ is not a Weierstrass point, we may take $D = g+1$.
\end{itemize}
\end{thm}
\begin{proof}
Step 0: If $g = 0$, then since $C(k) \neq \varnothing$ we have $C \cong \P^1$ and (Hilbertian fields are infinite)
$C$ has infinitely many closed points of degree $d$ for all $d \in \Z^+$.  Henceforth we suppose $g \geq 1$.

\medskip
\noindent Step 1: Let $K$ be a canonical divisor for $C$, and let $d \geq 2g-1$.  Then $\deg(K-d[P]) < 0$, so the Riemann-Roch space
$H^0 \mathcal{L}(n[P])$ has dimension $d-g+1$.  Thus $H^0 \mathcal{L} ((d+1)[P]) \supsetneq H^0 \mathcal{L}((d)[P])$, so for all $d \geq 2g$ there
is $f \in k(C)$ with polar divisor $d[P]$, so $f\colon C \ra \P^1$ is a morphism of degree $d$.  If $P$ is not a Weierstrass point
then $\dim H^0 \mathcal{L}(d[P])  = 1$ for all $1 \leq d \leq g$, and thus  $\dim H^0 \mathcal{L}(d[P]) = d+1-g$ for all
$d \geq g$.  As above, there is a degree $d$  morphism $f\colon C \ra \P^1$ for all $d \geq g+1$.

\medskip\noindent
Step 2: Let $f\colon C \ra \P^1$ be a morphism of degree $d$.  Let $l$ be the maximal separable subextension of $k(C)/k(\PP^1)$.
Applying the Hilbert Irreducibility Theorem to the Galois closure $m$ of $l/k(\PP^1)$, we see that there are infinitely many
$Q \in \PP^1(k)$ that are inert in $m/k(\P^1)$, hence also inert in $l/k(\P^1)$.  The extension $k(C)/l$ is purely inseparable,
so every place of $l$ is inert in $k(C)$ (see e.g. \cite[Lemma 14.20]{CA} for the affine case of this, to which we may
reduce).  Thus there are infinitely many $Q \in \P^1$ which are inert in $k(C)/k(\P^1)$, and this yields infinitely many
degree $d$ closed points on $C$.
\end{proof}

\begin{cor}
\label{NAJMANCOR2}
\label{COR2}
Let $N \in \Z^+$.  Then for all $d \geq \max\{1,2g(X_1(N))\}$, there is a sequence $\{(F_n,E_n,P_n)\}_{n=1}^{\infty}$ such that
$F_n$ is a number field of degree $d$, $(E_n)_{/F_n}$ is an elliptic curve, $j(E_n) \neq j(E_m)$ for all $m \neq n$ and $P_n \in E_n(F_n)$ is a point of order $N$ and $(E_n,P_n)$ cannot be defined over any proper subfield of $F_n$.
\end{cor}
\begin{proof}
Let $d \geq \max\{1,2g(X_1(N))\}$.  We apply Theorem \ref{THM1} with $k = \Q$, $C = X_1(N)$ and $P$ the cusp at infinity to get
a sequence $\{x_n\}_{n=1}^{\infty}$ of distinct degree $d$ closed points on $X_1(N)$.  Since $X_1(N) \ra X(1)$ is a finite map, by passing
to a subsequence we may assume that $x_n \in Y_1(N)$ for all $n$ and that $j(x_n) \neq j(x_m)$ for all
$m \neq n$.  Let $F_n$ be the residue field of $x_n$.  Since $Y_1(N)$ is a fine moduli space for $N \geq 4$, each $x_n$ corresponds to a unique pair $(E_n,P_n)$.  When $N \leq 3$ we use that the obstruction to being defined over the field of moduli vanishes
for closed points on modular curves.  Since $F_n$ is the field of moduli of $(E_n,P_n)$, the pair cannot be defined over a proper subfield of $F_n$.
\end{proof}

\begin{remark}
By \cite[Thm. 1]{KK96}, we have
\[ g(X_1(N)) = \begin{cases} 0 & \text{if $1 \leq N \leq 4$,} \\ 1 + \frac{N^2}{24}\prod_{p \mid N} \left(1-\frac{1}{p^2} \right) -
\frac{1}{4} \sum_{d \mid N} \varphi(d) \varphi(\frac{N}{d}) & \text{if $N \geq 5$.}
\end{cases}\]
Thus for all $N \in \Z^+$, we have $g(X_1(N))\le N^2/24 + 1$, so Corollary \ref{NAJMANCOR2} implies that for every degree $d \ge N^2/12 + 2$, there is an elliptic curve $E$ over a degree $d$ number field $F$ for which $N \mid \exp E(F)[{\rm tors}]$. Moreover, $E$ can be taken to not have CM. Indeed, our examples $E_{/F}$ have infinitely many distinct $j$-invariants, whereas there are only finitely many CM $j$-invariants of any given degree. Consequently, defining $T_{\neg{\rm CM}}(d)$ as the supremum of $\#E(F)[\rm tors]$ for a non-CM elliptic curve $E$ defined over a degree $d$ number field $F$, we have
\[ \liminf_d \frac{T_{\neg{\rm CM}}(d)}{\sqrt{d}} > 0. \]
On the other hand, by \cite[Thm. 6.4]{TRUTHII} we have
\[ \limsup_d \frac{T_{\neg{\rm CM}}(d)}{\sqrt{d \log \log d}} \geq \sqrt{ \frac{\pi^2 e^{\gamma}}{3}}. \]
It is natural to conjecture that $\limsup_d \frac{T_{\neg{\rm CM}}(d)}{\sqrt{d \log \log d}} < \infty$: whether this holds
was raised as a question by Hindry-Silverman \cite{HS99}. If yes, then $T_{\neg{\rm CM}}(d)$ grows remarkably regularly, in sharp contrast with the CM case (compare eq. \eqref{eq:truthIIeq} with Theorem \ref{MOTIVETHM}).
\end{remark}

\subsection{The proof of Theorem \ref{NAJMANTHM}}
\noindent
Let $\mathcal{A}(1)$ be the family of all elliptic curves.  Corollary \ref{COR2} gives: for all $B \in \Z^+$,
$\mathcal{S}(\mathcal{A}(1),B)$ contains all but finitely many elements of $\Z^+$, hence has density $1$.



\subsection{Arbitrary varieties}
\noindent In the setting of Theorem \ref{THM1}, without the assumption that $C(k) \neq \varnothing$, the conclusion of Theorem \ref{THM1}
need not hold: e.g., of $C_{/k}$ is a nice curve of genus zero with $C(k) = \varnothing$, then $C$ has no
closed points of any odd degree.

\medskip\noindent To go further, recall that the \textbf{index} $I(V)$ of a nice variety $V_{/k}$ is the least positive degree of a $k$-rational zero-cycle on $V$ -- equivalently, the greatest common divisor of all degrees of closed points on $V$.   Here is an analogue of
Theorem \ref{THM1} in the absence of the assumption $C(k) \neq \varnothing$, generalized from curves to all varieties.



\begin{thm}
\label{MONSTERNAJMANTHM}
Let $k$ be a Hilbertian field, and let $V_{/k}$ be a nice variety of positive dimension and index $I$.  There is $D \in \Z^+$ such that for
all $d \geq D$, the variety $V$ has infinitely many closed points of degree $dI$.
\end{thm}
\begin{proof}
Step 1: Suppose $C_{/k}$ is a nice curve of genus $g$.  The case $g = 0$ has already been treated above, so suppose $g \geq 1$.
Let $P_1,\ldots,P_r$ be distinct closed points of $C$ of degrees $d_1,\ldots,d_r$ such that $\gcd(d_1,\ldots,d_r) = I$.  There is a
positive integer $M$ such that if an integer $d$ is at least $M$ and divisible by $I$, then there are $n_1,\ldots,n_r \in \N$
such that $d = n_1 d_1 + \ldots + n_r d_r$: see e.g. \cite[Thm. 1.0.1]{Ramirez-Alfonsin}.  Let
\[ D = \max\{M, 2g-1 + \sum_{i=1}^r d_i\}. \]
Let $d$ be an integer such that $d \geq M$ and $I \mid d$, so we may write $d = \sum_{i=1}^r n_i d_i$ for $n_i \in \N$, and
we may assume without loss of generality that $n_1 \geq 1$.   The divisor $(n_1-1)[P_1] + \sum_{i=2}^r n_i [P_i]$ has
degree at least $2g-1$, so by Riemann-Roch we have
\[ \dim H^0 \mathcal{L} \left(\sum_{i=1}^r n_i [P_i]\right) = d_1 + \dim H^0 \mathcal{L} \left((n_1-1) [P_1] + \sum_{i=2}^r n_i [P_i] \right),\]
so there is a rational function on $C$ with polar divisor $\sum_{i=1}^r n_i [P_i]$ and thus a morphism $f\colon C \ra \P^1$
of degree $d$.  Step 2 of the proof of Theorem \ref{THM1} again applies to show that $C$ admits infinitely many closed points
of degree $d$.

\medskip\noindent
Step 2: Let $V_{/k}$ be a nice variety of index $I$.  There are closed points $P_1,\ldots,P_r$ of $V$ such that
$I = \gcd(\deg(P_1),\ldots,\deg(P_r))$; let $Z$ be the zero-cycle $\sum_{i=1}^r [P_i]$.  Then there is a nice curve $C_{/k}$
with $Z \subset C \subset V$: again, Hilbertian fields are infinite, so one may apply an extension of the Bertini Theorem
due to Kleiman-Altman \cite{Kleiman-Altman79}.  It follows that $C$ has index $I$, and applying Step 1 to $C$ finishes the proof.
\end{proof}
\noindent
Since by definition a variety can only have a closed point of degree $d$ if $d$ is a multiple of the index, Theorem \ref{MONSTERNAJMANTHM}
determines the set of degrees of closed points on any nice variety over a Hilbertian field \emph{up to a finite set}.

\begin{remark}
Let $C_{/k}$ be a nice curve, and let $f\colon C \ra \P^1$ be a morphism of degree $d$.  Pulling back a rational point on $\P^1$
yields an effective divisor of degree $d$ and thus $I(C) \mid d$.  The \textbf{gonality} $\operatorname{gon}(C)$ of $C$ is the least degree of a finite morphism $C \ra \P^1$, so we have
\[ I(C) \mid \operatorname{gon}(C). \]
We call a curve \textbf{$I$-gonal} if $I(C) = \operatorname{gon}(C)$.  Above we saw every genus zero curve is $I$-gonal
and the degrees of closed points on a genus zero curve are precisely the positive integer multiples of $I(C)$.  The same
argument holds for all $I$-gonal curves.  In some sense, ``most curves over most fields'' are $I$-gonal: if $g \geq 2$, $(\mathcal{M}_g)_{/\Q}$ is the moduli scheme of genus $g$ curves, $k = \Q(\mathcal{M}_g)$ is its function field and $C_{/k}$ is a nice curve
of genus $g$ with field of moduli $k$, then $I(C) = \operatorname{gon}(C) = 2g-2$.
\end{remark}

\begin{remark}\mbox{ }
\begin{enumerate}
\item[a)]
 Let $\F_q$ be a finite field, and let $C_{/\F_q}$ be a nice curve.  The Riemann Hypothesis for $C$ gives $C(\F_{q^d}) \neq \varnothing$ for all sufficiently large $d$.  In particular $C$ has index $1$, a
result of F.K. Schmidt.  If $V_{/\F_q}$ is a nice variety of positive dimension, then by Poonen's finite field
Bertini Theorem \cite{Poonen04}, $V$ contains a nice curve $C_{/\F_q}$.  So $V_{/\F_q}$ admits at least one
closed point of degree $d$ for all sufficiently large $d$.
\item[b)] Some hypothesis on $k$ is necessary to force a nice variety $V_{/k}$ to have closed points of degree
$dI(V)$ for all large $d$: e.g. $k$ must have extensions of these degrees!
\end{enumerate}
\end{remark}

\section{A criterion for typical boundedness}
\noindent
If $r$ and $s$ are nonzero rational numbers, we write $r \mid s$ if $\frac{s}{r} \in \Z$.

\subsection{Conditions (P1) and (P2)}
\noindent Consider the following conditions on a family $\mathcal{F}$ of abelian
varieties:

\medskip\noindent\textbf{Condition (P1):} for all primes $p$ and all $N \in \Z^+$, there is $n = n(p,N) \in \Z^+$ such that for all $A_{/F} \in \mathcal{F}$, if $A(F)$ has a point of order $p^n$ then $p^N \mid [F:\Q]$.

\medskip\noindent
\textbf{Condition (P2):} there is $c = c(\mathcal{F}) \in \Z^+$
such that for all prime numbers $p$ and all $A_{/F} \in \mathcal{F}$,
if $A(F)$ has a point of order $p$ then
\[ \frac{p-1}{c} \mid  [F:\Q]. \]

\begin{remark}
\label{REMARK10} \label{REMARK3.1}
For any finite set $\{p_1,\ldots,p_n\}$ of prime numbers, by
taking $c$ to be divisible by $\prod_{i=1}^n (p_i-1)$ we get
Condition (P2) for these primes automatically.  Thus Condition (P2)
holds if there is a positive integer $\tilde{c}$ and a number $P$
such that for all primes $p \geq P$ and all $A_{/F} \in \mathcal{F}$,
if $A(F)$ has a point of order $P$ then $p-1 \mid \tilde{c} [F:\Q]$.
\end{remark}

\begin{thm}
\label{THM4}
\label{P1P2THM}
Let $\mathcal{F}$ be a family of abelian varieties over number fields.
Suppose $\mathcal{F}$ satisfies Conditions (P1) and (P2).  Then:
\begin{itemize}
\item[a)] The exponent of
the torsion subgroup is typically
bounded in $\mathcal{F}$. That is, for each $\epsilon > 0$, there is $B_{\epsilon} \in \Z^{+}$ such that the set
\[
\{ d \in \Z^+ \mid \exists \  A_{/F} \in \mathcal{F} \text{ such that }
[F:\Q] = d \text{ and } \exp{A(F)[\tors]} \geq B_{\epsilon} \}
\]
has upper density at most $\epsilon$.
\item[b)] For $g \in \Z^+$, let $\mathcal{F}_g$ be the elements of $\mathcal{F}$
of dimension $g$.  Then the torsion subgroup is typically bounded in $\mathcal{F}_g$.
\end{itemize}
\end{thm}
\begin{proof}
a) We will use the Erd\H os-Wagstaff Theorem \cite[Thm. 2]{Erdos-Wagstaff80}: for all $\epsilon > 0$,
there is $C_{\epsilon}$ such that the set of positive integers
admitting a divisor of the form $\ell - 1 > C_{\epsilon}$ for a prime
number $\ell$ has upper density at most $\epsilon$.  It follows easily
that for any fixed $c \in \Z^+$, there is $C = C(\epsilon,c)$ such that the
set of positive integers $d$ for which there exists a prime
$\ell > C$ such that $\frac{\ell -1}{c} \mid d$ has upper density at most $\epsilon$.
By Condition (P2),
there is $c \in \Z^+$ such that for all $d \in \Z^+$, if $F/\Q$
is a number field of degree $d$ and $A_{/F} \in \mathcal{F}$ with $\ell \mid \# A(F)[\tors]$, we have
\[ \frac{\ell - 1}{c} \mid  d. \]
Thus, after removing a set of degrees $d$ of upper density at most
$\epsilon/2$, we get $L \in \Z^+$ such that if $\ell \mid \# A(F)[\tors]$ for some $A_{/F} \in \mathcal{F}$ with $[F:\Q] = d$,
then $\ell \leq L$.  For each $\ell \leq L$ and $N \in \Z^+$, the
set of $d$ which are divisible by $\ell^N$ has density $\ell^{-N}$, so if
$N$ is sufficiently large then the set of positive integers $d$
which is divisible by $\ell^N$ for some $\ell \leq L$ has density
at most $\frac{\epsilon}{2}$.  By Condition (P1), there is a positive
integer $n$ such that if $\ell \leq L$ and $A_{/F} \in \mathcal{F}$
has a point of order $\ell^{n+1}$ then $\ell^N \mid [F:\Q]$.  This gives a
set $\mathcal{D}_{\epsilon} \subset \Z^+$ of upper density at most
$\epsilon$ such that if $A_{/F} \in \mathcal{F}$ and $[F:\Q] \notin
\mathcal{D}_{\epsilon}$, then the order of any torsion point on
$A(F)$ divides $\prod_{\ell \leq L} \ell^n$. So we may take $B_{\epsilon}=1+\prod_{\ell \leq L} \ell^n$.

\medskip\noindent
b) If $A_{/F}$ is an abelian variety defined over a number field,
then \[\# A(F)[\tors] \mid (\operatorname{exp} A(F)[\tors])^{2\dim A}. \]
The result follows immediately from this and from part a).
\end{proof}


\subsection{Merelian families}
\noindent A family $\mathcal{F}$ of abelian varieties is
\emph{Merelian} if for all $d \in \Z^+$, there is $B(d) \in \Z^+$
such that for all number fields $F$ of degree $d$, if $A_{/F} \in \mathcal{F}$ then $\# A(F)[\tors] \leq B(d)$.

\begin{example}\mbox{ }
\begin{itemize}
\item[a)] (Merel \cite{Merel96}) The family $\mathcal{E} = \mathcal{A}(1)$ of all elliptic curves is Merelian.
\item[b)] It is believed that for all $g \geq 2$ the family $\mathcal{A}(g)$ of all $g$-dimensional abelian varieties is Merelian, but this seems to lie far out of reach.  \item[c)] (Silverberg \cite{Silverberg88}) Fix $g \geq 1$.  The family $\mathcal{A}_{\CM,g}$ of $g$-dimensional abelian varieties with complex multiplication is Merelian.
\item[d)] Work of Clark-Xarles \cite{CX08} produces Merelian families $\mathcal{F}_g$ with
\[ \mathcal{A}_{\CM,g} \subsetneq \mathcal{F}_g \subsetneq \mathcal{A}_g. \]
For instance,
one can take for $\mathcal{F}_g$ the family of all $g$-dimensional $A_{/F}$ such that $F$ admits a place $v_2 \mid 2$ and a place $v_3 \mid 3$ such that at each of the two places the N\'eron special fiber contains no copy of the
multiplicative group $\Gm$.
\end{itemize}
\end{example}

\begin{thm}
\label{6}
Let $\mathcal{F}$ be a family of abelian varieties over number fields that is closed under base extension, and let $\mathcal{F}_{g,d}$ be the subfamily of $\mathcal{F}_g$
consisting of all $g$-dimensional abelian varieties $A_{/F}$ such that there is a subfield
$F_0 \subset F$ with $[F_0:\Q] = d$ and an abelian variety $(A_0)_{/F_0}$
such that $(A_0)_{/F} \cong A_{/F}$.  If $\mathcal{F}_{g,d}$ is Merelian, then it satisfies Condition (P1).
\end{thm}

\begin{proof} Fix a prime number $p$ and a positive integer $N$.  If $A_{/F} \in \mathcal{F}_{g,d}$,
then $\dim A = g$ and
there is a subfield $F' \subset F$ with $[F':\Q] = d$ such that $A$
has a model over $F'$.  Let $P \in A(F)$ have order $p^n$ for some
$n \geq N$.  Then $F'(P) \subset F'(A[p^n])$, so
\[ [F'(P):F'] \mid [F'(A[p^n]):F'] \mid \# \GL_{2g}(\Z/p^n\Z).  \]
Moreover, for a certain integer $c$ with $\gcd(c,p)=1$, depending only on $p$ and $g$ (and not on $n$), we can write
\[ \# \GL_{2g}(\Z/p^n\Z)= cp^G \]
for some positive integer $G$.
On the other hand, because $\mathcal{F}_{g,d}$ is Merelian, if $n$ is sufficiently large
compared to $N$ then
\[ [F'(P):F'] \geq c p^N . \]
It follows that
\[ p^N \mid [F'(P):F'] \mid [F:\Q]. \qedhere\]
\end{proof}

\begin{thm}
\label{THM7}
\label{7}
Let $d_0 \in \Z^+$.  The family $\mathcal{E}_{d_0}$ of all elliptic
curves $E_{/F}$ defined over number fields such that $[\Q(j(E)):\Q] = d_0$
 satisfies Condition (P1).
\end{thm}
\begin{proof}
Combining Theorem \ref{6} with Merel's Theorem shows: the
family of elliptic curves arising by base extension from
a number field of degree $d_0$ satisfies Condition (P1).
Fix $p$ a prime number, $N$ a positive integer, and let $n \in \Z^+$
be such that for every elliptic curve $E_{/F}$ arising by base
extension from a number field of degree $d_0$ such that
$E(F)$ has a point of order $p^n$, we have $p^{N+2} \mid [F:\Q]$.  \\ \indent
Let $F$ be any number field, let $E_{/F}$ be an elliptic curve
with $[\Q(j(E)):\Q] = d_0$, and suppose that $E(F)$ has a point of
order $p^n$.  Let $F_0 = \Q(j(E))$, let $(E_0)_{/F_0}$ be
any elliptic curve.  Then $F_0 \subset F$, and there is
$F'/F$ of degree dividing $12$ such that $E_{/F'} \cong (E_0)_{/F'}$.
Then $E_0(F') \cong E(F') \supset E(F)$ has a point of order $p^n$,
so
\[ p^{N+2} \mid [F':\Q] = [F':F][F:\Q] \mid 12 [F:\Q], \]
and thus
\[ \ord_p [F:\Q] \geq N. \qedhere \]
\end{proof}

\section{The proofs of Theorems \ref{BIGCMTHM}, \ref{FZEROTHM} and \ref{DZEROTHM}}

\subsection{The proof of Theorem \ref{BIGCMTHM}}

\begin{thm}[Gaudron-R\'emond {\cite{GR17}}]
\label{GRTHM}
Let $A_{/F}$ be an abelian variety defined over a number field.  Suppose $A$ has complex multiplication and $\End_F A = \End A$.
Let $\mu$ be the number of roots of unity in the center of $\End A$.  Let $e$ be the exponent of $A(F)[\tors]$.  Then
\begin{equation}
\label{GAUDRONREMONDEQ}
 \varphi(e) \mid \frac{\mu}{2} [F:\Q].
\end{equation}
\end{thm}

\begin{cor}
\label{GRCOR}
Let $g \in \Z^+$.  There is a positive integer $c(g)$ such that: for all number fields $F$ and all $g$-dimensional CM
abelian varieties $A_{/F}$, if $A(F)$ has a rational point of order $N$, then
\[ \varphi(N) \mid c(g) [F:\Q]. \]
\end{cor}
\begin{proof}
Step 1: There is $H(g) \in \Z^+$ such that for every $g$-dimensional abelian variety $A$ defined over a
field $F$ of characteristic $0$, there is a field extension $K/F$ with $[K:F] \mid H(g)$ such that $\End_K A_{/K} =
\End A$.  This can be proved by observing that $\End A \cong \Z^d$ for some $1 \leq d \leq 2g^2$ and thus the action of
$\gg_F$ on $\End A$ is given by a homomorphism $\gg_F \ra \GL_d (\Z)$ with finite image, and using a result of Minkowski \cite{Minkowski87} that  there is an absolute bound on the size of a finite subgroup of $\GL_d(\Z)$.\footnote{This does not lead to a very good bound.   Better bounds occur as a special case of work of Silverberg \cite{Silverberg92b}.  Recent
work of Guralnick-Kedlaya \cite{Guralnick-Kedlaya16} computes the optimal value of $H(g)$ for all $g$, and recent work of R\'emond \cite{Remond17}
computes the maximal value of $[K:F]$ for each $g$.}  We get a variant of Theorem \ref{GRTHM} in which the hypothesis
$\End_F A = \End A$ is dropped and the conclusion is $\varphi(e) \mid \frac{H(g) \mu}{2} [F:\Q]$.

\medskip \noindent
Step 2: We claim that there is $M(g) \in \Z^+$ such that if $A_{/F}$ is a $g$-dimensional CM abelian variety with
$\End_F A = \End A$ then the number of roots of unity in the center of $\End A$ divides $M(g)$.   Let $Z$ be the center of
$\End A$ and $Z^0$ be the center of $\End^0 A$.  Then $Z^0$ is an \'etale $\Q$-algebra of dimension $2g$
\cite[Prop. 1.3 and $\S$3]{Milne}, i.e., there are number fields $F_1,\ldots,F_r$ such that $Z^0 \cong \prod_{i=1}^r F_i$
and $2g = \dim Z^0 = \prod_{i=1}^r [F_i:\Q]$.  Thus $Z$ is isomorphic to a subring of $\prod_{i=1}^r \Z_{F_i}$,
where $\Z_{F_i}$ is the ring of integers of $F_i$, so if $\mu_i$ is the number of roots of unity in $F_i$ then
$\mu \mid \prod_{i=1}^r \mu_i$.  The number of roots of unity in a number field $F$ is bounded in terms of
$[F:\Q$] -- indeed, if this number is $N$, then $F \subset \Q(\zeta_N)$, so $\varphi(N) \mid [F:\Q]$ and $N = O([F:\Q] \log \log [F:\Q])$.
Moreover $r \leq 2g$ and $[F_i:\Q] \leq 2g$ for all $i$.  This establishes the claim, and the result follows.
\end{proof}
\noindent
Let $p$ be a prime number, let $N \in \Z^+$, and let $A_{/F}$ be a $g$-dimensional CM abelian variety defined over a number field.
Using Corollary \ref{GRCOR}, we get:
\begin{itemize} \item If $N \in \Z^+$, if $A(F)$ has a point of order $p^{N+\ord_p(c(g))+1}$ then
\[ \varphi(p^{N+\ord_p(c(g))+1}) = (p-1) p^{N+\ord_p(c(g))} \mid c(g) [F:\Q], \]
so $p^N \mid [F:\Q]$.  So Condition (P1) holds for the family $\mathcal{A}_{\CM}(g)$.
\item If $A(F)$ has a point of order $p$ then $p-1 = \varphi(p) \mid c(g)[F:\Q]$, establishing Condition (P2) for the family $\mathcal{A}_{\CM}(g)$.
\end{itemize}
By Theorem \ref{P1P2THM}, torsion is typically bounded on $\mathcal{A}_{\CM}(g)$.

\subsection{The proofs of Theorems \ref{FZEROTHM} and \ref{DZEROTHM}}
\noindent Theorem \ref{THM7} reduces the proofs of Theorems \ref{FZEROTHM} and \ref{DZEROTHM} to the
following result.

\begin{thm}\mbox{ }
\label{LASTTHM}
\begin{itemize}
\item[a)] Let $d_0 \in \Z^+$ be such that $\operatorname{SI}(d_0)$ holds.  The family $\mathcal{E}_{d_0}$
of all elliptic curves defined over number fields such that $[\Q(j(E)):\Q] = d_0$ satisfies condition (P2).
\item[b)] Let $F_0$ be a number field that does not contain the Hilbert class field of any imaginary quadratic field.  If $[F_0:\Q] \geq 3$ then
we assume the Generalized Riemann Hypothesis (GRH).  Then the family $\mathcal{E}_{F_0}$ of all elliptic curves defined over a number field $F \supset F_0$ such that $j(E) \in F_0$ satisfies condition (P2).
\end{itemize}
\end{thm}
\noindent
We begin the proof of Theorem \ref{LASTTHM}.  In view of the work of \cite{BCP17} (or Theorem \ref{BIGCMTHM}) we may restrict attention to elliptic curves \emph{without} CM.  We will prove part a) and then discuss the (minor) modifications
necessary to prove part b).
\\ \\
Let $d_0 \in \Z^+$ be such that $\operatorname{SI}(d_0)$ holds.  For $E_{/F} \in \mathcal{E}_{d_0}$, let $F_0 = \Q(j(E))$ and let $(E_0)_{/F_0}$ be an elliptic curve with $j(E_0) = j(E)$, so $F_0 \subset F$.  Since $E$ has no CM, there is a quadratic extension $F'/F$ such that
$E_{/F'} \cong (E_0)_{/F'}$.   Since $[F':\Q] = 2[F:\Q]$, if the family $\mathcal{G}_{d_0}$ of all elliptic curves $E_{/F}$ that
arise by base extension from $(E_0)_{/\Q(j(E_0))}$ with $[\Q(j(E_0)):\Q] = d_0$ satisfies condition (P2) for some $c \in \Z^+$, then the family $\mathcal{E}_{d_0}$ satisfies (P2) for $2c$.  So we may work with  $\mathcal{G}_{d_0}$.
\\ \indent
Let $F_0$ be a number field of degree $d_0$, and let $(E_0)_{/F_0}$ be an elliptic curve with $\Q(j(E)) = F_0$.   Let $p$ be a prime number,  let
\[ \rho_p\colon \gg_{F_0} \ra \GL(E_0[p]) \]
be the mod $p$ Galois representation attached to $(E_0)_{/F_0}$, let $G = \rho_p(\gg_{F_0})$ be its image, and let
$\overline{G}$ be its projective image, i.e., the image of $G$ under the homomorphism $\GL(E_0[p]) \ra \PGL(E_0[p])$.  The degrees of extensions $F/F_0$ such that $E_0(F)$ has a point of order $p$ are the multiples of the sizes of the orbits of $G$
on $E_0[p] \setminus \{0\}$.  Thus it is enough to show that there is $c = c(d_0)$ such that for all $(E_0)_{/F_0}$ as above and for all nonzero $P \in E_0[p]$, we have
\[ \frac{p-1}{c} \divides [F_0(P):\Q]. \]
Recall that $\det \rho_p\colon \gg_{F_0} \ra (\Z/p\Z)^{\times}$ is the mod $p$ cyclotomic character $\chi_p$.
Let \[i(p) = [(\Z/p\Z)^{\times}:\chi_p(\gg_{F_0})]. \]  Then $i(p)$ depends on $F_0$ and not just $d_0$, but harmlessly: since $\chi_p(\gg_{\Q}) = (\Z/p\Z)^{\times}$,
we have $i(p) \mid d_0$.  In particular, we have in all cases that
\[ \frac{p-1}{d_0} \divides \frac{p-1}{i(p)} \divides \# G. \]
\\
A choice of $\F_p$-basis $e_1,e_2$ for $E_0[p]$ induces an isomorphism $\GL(E_0[p]) \stackrel{\sim}{\ra} \GL_2(\F_p)$, thus
$G$ may be viewed as a subgroup of $\GL_2(\F_p)$, well-defined up to conjugacy.  We recall the classification of subgroups of $\GL_2(\F_p)$ -- essentially due to Dickson, and given in \cite[\S2.4, \S2.6]{Serre72}: for $G \subset \GL_2(\F_p)$, at least one of the following holds:
\begin{itemize}
\item[(1)] The subgroup $G$ contains $\SL(E_0[p])$.
\item[(2)] The subgroup $G$ is contained in the normalizer of a split Cartan.
\item[(3)] The subgroup $G$ is contained in the normalizer of a nonsplit Cartan.
\item[(4)] The subgroup $\overline{G}$ is isomorphic to $A_4$, $S_4$ or $A_5$.
\item[(5)] The subgroup $G$ is contained in a Borel.
\end{itemize}
We will address each of these cases separately.
\subsubsection{Case 1} Suppose $G = \rho_p(\gg_{F_0})$ contains $\SL_2(E_0[p])$.  Then \[ [\GL_2(E_0[p]):G] \mid i(p) \mid d_0.\]
Since $\GL(E_0[p])$ acts transitively on $E_0[p] \setminus \{0\}$, the size of any orbit of $G$ on $E_0[p] \setminus \{0\}$ is
divisible by \[ \frac{\# (E_0[p] \setminus \{0\})}{d_0} = \frac{(p+1)(p-1)}{d_0}.\]  Thus $\frac{p-1}{d_0} \mid [F_0(P):F_0]$,
so $p-1 \mid [F_0(P):\Q]$: we may take $c = 1$.
\subsubsection{Case 2} A split Cartan subgroup $C_s$ of $\GL_2(E_0[p])$ is the set of matrices $m \in \GL_2(\F_p)$ that pointwise
fix each of a pair $L_1,L_2$ of one-dimensional subspaces of $E_0[p]$.  Its normalizer $NC_s$ is the set of matrices $m \in \GL_2(\F_p)$
that stabilize the pair $\{L_1,\L_2\}$, i.e., either $m L_1 = L_1$ and $m L_2 = L_2$ (so $m \in C_s$) or $mL_1 = L_2$ and $m L_2 =
L_1$ (so $m \notin C_s$).  Choosing a basis such that $e_1$ generates $L_1$
and $e_2$ generates $L_2$, these groups are given explicitly as
\[ C_s = \left\{ \left( \begin{array}{cc} a & 0 \\ 0 & b \end{array} \right) \bigg{\vert} \ a,b \in \F_p^{\times} \right\}, \]
\[ N C_s = C_s \coprod \left\{ \left( \begin{array}{cc} 0 & c \\ d & 0 \end{array} \right) \bigg{\vert} \ c,d \in \F_p^{\times} \right\}. \]
In particular, if $G$ is contained in a split Cartan, then it is contained a Borel subgroup, namely
the set of matrices $m \in \GL_2(E_0[p])$ that fix a single line $L$.  This is taken care of in Case 5 below.  So suppose
that $G$ is contained in $NC_s$ but not contained in $C_s$.  Let $F \supset F_0$ be a number field such that $E_0(F)$ has a point $P$ of order $p$.
Then $\left.\rho_p\right|_{\gg_F}$ has image contained in $NC_s$ and also fixes $P$ and thus the one-dimensional subspace $L = \langle P \rangle$, so the following result applies with $H = \rho_p(\gg_F)$.

\begin{lemma}[{Lozano-Robledo \cite[Lemma 6.6]{LR13}}]
\label{LRLEMMA6.6}
Let $H$ be a nontrivial subgroup of $NC_s$ that fixes each element of a one-dimensional subspace $L$ of $E_0[p]$.  Then
one of the following holds:
\begin{itemize}
\item[(i)] $H$ is contained in the subgroup $\left\{ \left( \begin{array}{cc} 1 & 0 \\ 0 & b \end{array} \right) \bigg{\vert} \ b \in \F_p^{\times}
\right\}$ and $V = L_1$.
\item[(ii)] $H$ is contained in the subgroup $\left\{ \left( \begin{array}{cc} a & 0 \\ 0 & 1 \end{array} \right) \bigg{\vert} \ a \in \F_p^{\times}
\right\}$ and $V = L_2$.
\item[(iii)] $H = \left\{ \left( \begin{array}{cc} 1 & 0 \\ 0 & 1 \end{array} \right), \left( \begin{array}{cc} 0 & c \\ c^{-1} & 0 \end{array}
\right) \bigg{\vert} \ \text{ for some $c$ in } \F_p^{\times} \right\}$ and $V = \langle (c,1) \rangle$.
\end{itemize}
\end{lemma}
\noindent
Suppose we are in Case (i) of Lemma \ref{LRLEMMA6.6}.  Let $g$ be a generator of the (unique, cyclic) subgroup of $\F_p^{\times}$
of order $\frac{p-1}{i(p)}$, and let $M_g \in G$ be such that $\det(M_g) = g$.

First suppose $M_g \in C_s$, say $M_g = \left( \begin{array}{cc} a & 0 \\ 0 & b \end{array} \right)$.  Since $G$ is not contained
in $C_s$, there is also a matrix $A \coloneqq \left( \begin{array}{cc} 0 & c \\ d & 0 \end{array} \right) \in G$, and so
$$ \begin{pmatrix}
0 & d^{-1} \\
c^{-1} & 0
\end{pmatrix}\begin{pmatrix}
a & 0 \\
0 & b
\end{pmatrix}\begin{pmatrix}
0 & c \\
d & 0
\end{pmatrix}=
\begin{pmatrix}
b & 0 \\
0 & a
\end{pmatrix} \in G$$
and thus also
$$\begin{pmatrix}
a & 0 \\
0 & b
\end{pmatrix}\begin{pmatrix}
b & 0 \\
0 & a
\end{pmatrix} =
\begin{pmatrix}
ab & 0 \\
0 & ab
\end{pmatrix} =\begin{pmatrix}
g & 0 \\
0 & g
\end{pmatrix} \in G.$$
\noindent
Let $Z \coloneqq \bigg{\langle} \left( \begin{array}{cc} g & 0 \\ 0 & g \end{array} \right) \bigg{\rangle}$.
Then $Z \subset G$ and $Z \cap H = \{e\}$, so
\[ \frac{p-1}{d_0} \divides  \frac{p-1}{i(p)}  \divides [G:H] = [F_0(P):F_0], \]
so
\[ p-1 \mid [F_0(P):\Q]. \]
\indent
Now suppose $M_g \in N C_s \setminus C_s$, say $M_g = \left( \begin{array}{cc} 0 & c \\ d & 0 \end{array} \right)$.
Then $\det M_g = -cd$, so $cd = -g$.  Also $M_g^2 = \left( \begin{array}{cc} -g & 0 \\ 0 & -g \end{array} \right) \in G$.
Let $W \coloneqq \bigg{\langle} \left( \begin{array}{cc} -g & 0 \\ 0 & -g \end{array} \right) \bigg{\rangle}$.  Then
$\frac{p-1}{2 i(p)} \mid \# W$ and $W \cap H = \{e\}$, so reasoning as above we get
\[ \frac{p-1}{2} \mid [F_0(P):\Q]. \]

The analysis for Case (ii) of Lemma \ref{LRLEMMA6.6} is exactly as for Case (i).

Suppose we are in case (iii) of Lemma \ref{LRLEMMA6.6}.  Then $\# \rho_p(\gg_F) \mid 2$.  As above, $d_0$ times the order of $\det \rho_p(\gg_{F_0})$ is divisible by $p-1$, so
\[ p-1 \mid d_0 [G:H] \mid 2 [F_0(P):\Q]. \]

Conclusion: in Case 2, we may take $c = 2$.

\subsubsection{Case 3}
A nonsplit Cartan subgroup $C_{\operatorname{ns}}$ of $\GL(E_0[p])$ is obtained by endowing $E_0[p]$ with the structure
of a $1$-dimensional vector space over $\F_{p^2}$ and taking the group of $\F_{p^2}$-linear automorphisms.  It is known
that all such groups are conjugate in $\GL(E_0[p])$ and that the normalizer of $C_{\operatorname{ns}}$ consists
of all $\F_{p^2}$ \emph{semi}linear automorphisms of $\GL(E_0[p])$ and has order $2(p^2-1)$.
But we have already recalled more than we need, in view of the following result (\cite[Lemma 7.4]{LR13}).

\begin{lemma}
\label{H-nonsplit}
Let $N C_{\operatorname{ns}}$ be the normalizer of a nonsplit Cartan subgroup of $\GL(E_0[p])$, and let $H \subset N C_{\operatorname{ns}}$
be a subgroup that fixes each element in a one-dimensional $\F_p$-subspace of $E_0[p]$.  Then $\# H \leq 2$.
\end{lemma}
\noindent
Therefore for a nonzero $P \in E_0[P]$, we have
\[ \frac{p-1}{2 i(p)} \mid \frac{p-1}{2 d_0} \mid [G:H] = [F_0(P):F_0], \]
so $\frac{p-1}{2} \mid [F_0(P):\Q]$ and we may take $c = 2$.


\subsubsection{Case 4} We will use the following result of Etropolski, based on work of Serre and an observation of Mazur.

\begin{prop}[{\cite[Prop. 2.6]{Etropolski16}}]
Let $E_{/K_0}$ be an elliptic curve defined over a number field $K_0$ of degree $d_0$.  For a prime number $p$, let $\overline{G}$ be the projective
image of the mod $p$ Galois representation.
\begin{itemize}
\item[a)] If $\overline{G} \cong A_4$, then $p \leq 9d_0 + 1$.
\item[b)] If $\overline{G} \cong S_4$, then $p \leq 12d_0 + 1$.
\item[c)] If $\overline{G} \cong A_5$, then $p \leq 15d_0 + 1$.
\end{itemize}
\end{prop}
\noindent
So this case can occur only if $p \leq 15d_0 + 1$; by Remark \ref{REMARK3.1} we can omit these primes.

\subsubsection{Case 5} The subgroup $G$ is contained in a Borel subgroup iff $E$ admits an $F_0$-rational $p$-isogeny.
Our assumption $\operatorname{SI}(d_0)$ is precisely that the set of primes $p$ for which a non-CM elliptic curve defined over a
number field of degree $d_0$ is finite, so by Remark \ref{REMARK3.1} we can omit these primes.  This completes the proof of part a).

\subsubsection{The proof of part b)} The hypothesis $\operatorname{SI}(d_0)$ was only used in Case 5, so only Case 5 needs
to be redone under the hypotheses of part b): suppose that $F_0$ does not contain the Hilbert class field of any imaginary quadratic field.  (This hypothesis
precisely prohibits having an elliptic curve $E_{/F_0}$ for which the CM is $F_0$-rationally defined.)  Then Larson-Vaintrob
showed \cite{LV14} that, assuming (GRH), the set of primes $p$ such that an elliptic curve $E_{/F_0}$ admits an $F$-rational
$p$-isogeny is bounded, giving Case 5 under (GRH).  Finally, suppose moreover that $[F_0:\Q] = 2$.  Then the hypothesis on $F_0$ becomes that $F_0$
is not itself an imaginary quadratic field of class number $1$, and under this hypothesis Momose showed \cite{Momose95}
that the set of primes $p$ such that an elliptic curve $E_{/F_0}$ admits an $F$-rational $p$-isogeny is finite.  This completes
the proof of Case 5 and thus the proof of Theorem \ref{LASTTHM}b).


\end{document}